\documentclass{amsart}
\usepackage[cp1250]{inputenc}
\usepackage{graphicx}
\usepackage{amsmath}
\usepackage{latexsym}
\usepackage{amssymb}
\usepackage{pstricks}
\usepackage{graphics,color}
\usepackage{color}
\usepackage{mathrsfs}
\newtheorem{theorem}{Theorem}[section]

\newtheorem{lemma}[theorem]{Lemma}

\theoremstyle{definition}
\newtheorem{definition}[theorem]{Definition}
\newtheorem{remark}{Remark}

\newcommand{\om}{\Omega}

\newcommand{\dom}{{\rm d}\Omega}

\newcommand{\dc}{\|}

\newcommand{\haa}{H}

\newcommand{\iom}{\int_\om}

\newcommand{\ldh}{L^2(0,T;\,H)}
\newcommand{\ldsh}{L^2(0,\ts;\,H)}

\newcommand{\nha}{\|_{\haa}}

\newcommand{\nt}{(0,T)}

\newcommand{\omnt}{\om\times\nt}

\newcommand{\ts}{T^*}

\newcommand{\unu}{{\bfu}_0}

\newcommand{\vajd}{V}
\newcommand{\vakp}{V^{k,p}}

\newcommand{\bfb}{\mbox{\boldmath{$b$}}}

\newcommand{\bfn}{\mbox{\boldmath{$n$}}}
\newcommand{\bff}{\mbox{\boldmath{$f$}}}

\newcommand{\bfz}{\mbox{\boldmath{$z$}}}

\newcommand{\bfu}{\mbox{\boldmath{$u$}}}
\newcommand{\bfw}{\mbox{\boldmath{$w$}}}
\newcommand{\bfv}{\mbox{\boldmath{$v$}}}

\newcommand{\bfmathcalF}{\mbox{\boldmath{$\mathcal{F}$}}}
\newcommand{\bfmathcalA}{\mbox{\boldmath{$\mathcal{A}$}}}

\newcommand{\bfphi}{\mbox{\boldmath{$\phi$}}}

\newcommand{\bfpsi}{\mbox{\boldmath{$\psi$}}}

\newcommand{\bfvarphi}{\mbox{\boldmath{$\varphi$}}}

\newcommand{\bftau}{\mbox{\boldmath{$\tau$}}}

\newcommand{\bfzero}{{\bf 0}}

\begin{document}

\title[Navier--Stokes Flows in Polyhedral Domains]{Mixed Initial-Boundary
Value Problem for the Three-Dimensional Navier--Stokes
Equations in Polyhedral Domains}%
\author{Michal Bene\v{s}$^{1,2}$}%
\address{$^{1}$Czech Technical University in Prague,
Faculty of Civil Engineering, Department of
Mathematics, Th\'{a}kurova 7, 166 29 Prague 6, Czech Republic}
\address{$^{2}$Centre for Integrated Design of Advanced Structures,
Th\'{a}kurova 7, 166 29 Prague 6, Czech Republic}
\email{benes@mat.fsv.cvut.cz}%

\thanks{This research was supported by the project GA\v{C}R P201/10/P396.
Additional support from the Ministry of Education, Youth and Sports
of the Czech Republic, project No. 1M0579, within activities of the
CIDEAS research centre is greatly acknowledged.}%

\date{01-02-2011}%

%\dedicatory{*}%
%\commby{*}%
% ----------------------------------------------------------------
\begin{abstract}
We study a mixed initial--boundary value problem  for the
Navier--Stokes equations, where the Dirichlet, Neumann and slip
boundary conditions are prescribed on the faces of a
three-dimensional polyhedral domain. We prove the existence,
uniqueness and smoothness of the solution on a time interval
$(0,T^*)$, where $0<T^*\leq T$.

\end{abstract}

\maketitle

% ----------------------------------------------------------------

\section{Introduction}

\subsection{Preliminaries}
We consider a mixed initial--boundary value problem for the
Navier--Stokes equations in a three-dimensional domain $\Omega
\subset \mathbb{R}^3$ of polyhedral type with a boundary $\partial
\Omega$.  The domain $\om$ represents e.g. a channel filled up by a
moving fluid. $\partial\Omega$ consists of nonintersecting pieces
$\Gamma_D$, $\Gamma_G$ and $\Gamma_N$, $\partial \Omega =
\overline{\Gamma_D} \cup \overline{\Gamma_G} \cup
\overline{\Gamma_N}$. $\Gamma_D = \bigcup_{j\in \mathcal{J}_1}
\Gamma_j$ represents solid walls, $\Gamma_G =
\bigcup_{j\in\mathcal{J}_2} \Gamma_j$ denotes uncovered fluid
surfaces and $\Gamma_N = \bigcup_{j\in\mathcal{J}_3} \Gamma_j$
denotes the artificial part of the boundary such as the exit (or a
free surface), $\mathcal{J}_1 \cup \mathcal{J}_2 \cup \mathcal{J}_3=
\left\{1,\dots,n\right\}$ and $\Gamma_i \cap \Gamma_j = \emptyset$
iff $i \neq j$, $i,j\in \left\{1,\dots,n\right\}$.

We study the existence and uniqueness of the solution $\bfu$ to the
Navier--Stokes flows on $(0,T)$, $T>0$, in $\om$ under the following
boundary conditions:
\begin{align}
\bfu &= {\bf 0}  &\textmd{ on }& \Gamma_D\times\nt, \label{Dirichlet} \\
\bfu \cdot \bfn=0, \;\; [\nabla\bfu+(\nabla\bfu)^{\top}]\bfn\cdot
\bftau &= 0 &\textmd{ on }&
\Gamma_G\times\nt, \label{Navier}\\
-\mathcal{P}\bfn+\nu[\nabla\bfu+(\nabla\bfu)^{\top}]\bfn &= {\bf0}
&\textmd{ on }& \Gamma_N\times\nt. \label{artificial}
\end{align}
In \ref{Dirichlet}--\ref{artificial} $\bfu=(u_1,u_2,u_3)$ and
$\mathcal{P}$ denote the unknown velocity and pressure,
respectively. Further, $\nu$ denotes the viscosity of the fluid.
$\bfn= (n_1,n_2,n_3)$ and $\bftau= (\tau_1,\tau_2,\tau_3)$ are unit
normal and tangent vectors, respectively, to $\partial\Omega$.

\subsection{The domain}\label{domain}
It is well known that the regularity results for solutions of
elliptic problems in domains with edges or with the mixed boundary
conditions are closely related to the properties of the boundary of
the domain. Hence we specify several attributes of the domain
$\Omega$, which will be used later. We assume that
\begin{enumerate}

\item[(i)] $\Gamma_i$ (the faces of
$\Omega$), $i=1,\dots,n$, are open two-dimensional manifolds of
class $C^\infty$;

\item[(ii)] the boundary $\partial\Omega$ consists of smooth
faces $\Gamma_i$ (defined above) and smooth (of class $C^\infty$)
nonintersecting curves $\mathcal{M}_k$ (the edges), $k=1,\dots,m$,
vertices on $\partial \Omega$ are excluded;

\item[(iii)] for every $A\in\mathcal{M}_k$, $k=1,\dots,m$,
there exists a neighborhood $\mathcal{U}_A$ and a diffeomorphic
mapping $\kappa_A$ which maps $\Omega \cap \mathcal{U}_A$ onto
$\mathcal{D}_A\cap B_A$, where $\mathcal{D}_A$ is a dihedron of the
form
$$
\left\{[x_1,x_2,x_3]\in\mathbb{R}^3; \; 0<r<\infty, \; -\omega_{A}/2
< \varphi < \omega_{A}/2, \; x_3\in\mathbb{R} \right\},
$$
$\omega_{A}>0$ denotes the angle at the edge $\mathcal{M}_k$,
$A\in\mathcal{M}_k$, and $B_A$ is the unit ball ($r,\varphi$ denote
the polar coordinates in the $(x_1,x_2)$-plane);

\item[(iv)] $\Gamma_i\in\Gamma_D$, i.e. $i\in
\mathcal{J}_1$, forms at least one of the adjoining faces of every
edge $\mathcal{M}_k$, $k=1,\dots,m$;

\item[(v)]
$\left\{
  \begin{array}{ll}

\hbox{for every } A\in\mathcal{M}_k,\; \mathcal{M}_k \subset
\overline{\Gamma_D}\cap\overline{\Gamma_D}, & \hbox{we have
}\omega_{A}<\pi,
\\
\hbox{for every } A\in\mathcal{M}_k,\; \mathcal{M}_k \subset
\overline{\Gamma_D}\cap\overline{\Gamma_G}, & \hbox{we have
}\omega_{A}<(3/4)\pi,
\\
\hbox{for every } A\in\mathcal{M}_k,\; \mathcal{M}_k \subset
\overline{\Gamma_D}\cap\overline{\Gamma_N}, &
\hbox{we have }\omega_{A}<(1/4)\pi. \\

  \end{array}
\right.$

\end{enumerate}

\subsection{Basic notation and some function spaces}
Vector functions and operators acting on vector functions are
denoted by~boldface letters. Unless specified otherwise, we use
Einstein's summation convention for indices running from $1$ to $3$.

For an arbitrary $r\in [1,+\infty]$, $L^r(\Omega)$ denotes the usual
Lebesgue space equipped with the norm $\|\cdot\|_{L^r(\Omega)}$, and
$W^{k,p}(\Omega)$, $k\geq 0$ ($k$ need not to be an integer, see
\cite{KufFucJoh1977}), $1\leq p < \infty$, denotes the usual Sobolev
space with the norm $\|\cdot\|_{W^{k,p}(\Omega)}$.

\smallskip

Let
\begin{displaymath}
{E}:=\left\{\bfu\in C^\infty(\overline{\Omega})^3;\;
\textmd{div}\,\bfu = 0, \, {\textmd{supp}\, \bfu}  \cap \Gamma_D =
\emptyset, \, {\textmd{supp}\, \bfu \cdot \bfn} \cap \Gamma_G =
\emptyset  \right\}
\end{displaymath}
and $\vakp$ be a closure of $E$ in the norm of $W^{k,p}(\om)^3$,
$k\ge 0$ ($k$ need not be an integer) and $1\leq p<\infty$. Then
$\vakp$ is a Banach space with the norm of the space
$W^{k,p}(\om)^3$. For simplicity, we denote $V^{1,2}$ and $V^{0,2}$,
respectively, as $\vajd$ and $\haa$. Note, that $\vajd$ and $\haa$,
respectively, are Hilbert spaces with scalar products
\begin{equation}\label{scalar_V}
((\bfu,\bfv))  = 2\iom \varepsilon_{ij}(\bfu)
\varepsilon_{ij}(\bfv)\,\dom
\end{equation}
and
\begin{equation*}
(\bfu,\bfv) = \iom u_i v_i \,\dom
\end{equation*}
and they are closed subspaces of spaces $W^{1,2}(\om)^3$ and
$L^2(\om)^3$, respectively. In \ref{scalar_V} $e_{ij}(\bfu)$ denotes
the matrix with the components
\begin{displaymath}
e_{ij}(\bfu)=\frac{1}{2}\left(\frac{\partial u_i}{\partial
x_j}+\frac{\partial u_j}{\partial x_i}\right).
\end{displaymath}

Further, define the space
\begin{equation}\label{space_D}
\mathcal{D}:= \left\{\bfu \;|\; \bff \in H,\; ((\bfu,\bfv))
=(\bff,\bfv) \textmd{ for all } \bfv\in V \right\}
\end{equation}
equipped with the norm
\begin{displaymath}
\|\bfu\|_{\mathcal{D}} := \dc\bff\nha,
\end{displaymath}
where $\bfu$ and $\bff$ are corresponding functions via
\ref{space_D}.

Let $\bfu,\,\bfv,\bfw\in W^{1,2}(\Omega)^3$. We will use the
notation
\begin{equation*} b(\bfu,\bfv,\bfw) =
\iom(\bfu\cdot\nabla)\bfv\cdot\bfw\,\dom.
\end{equation*}

Throughout the paper, we will always use positive constants $c$,
$c_1$, $c_2$, $\dots$, which are not specified and which may differ
from line to line.

% ----------------------------------------------------------------
\section{Formulation of the problem}
 Let $T\in(0,\infty)$, $Q = \omnt$. The
classical formulation of our problem is as follows:
\begin{align}
\bfu_t - \nu\Delta \bfu + (\bfu\cdot\nabla)\bfu + \nabla \mathcal{P}
&= \bff
&&\textmd{ in }\; Q, \label{eq3} \\
{\rm div}\,\bfu &= 0 &&\textmd{  in } \; Q, \label{eq5} \\
\bfu &= {\bf 0}  &&\textmd{  on }\; \Gamma_D\times\nt, \label{eq6} \\
\bfu \cdot \bfn=0, \;\; [\nabla\bfu+(\nabla\bfu)^{\top}]\bfn\cdot
\bftau  &= 0 &&\textmd{ on }\;
\Gamma_G\times\nt,\label{eq6b}\\
-\mathcal{P}\bfn+\nu[\nabla\bfu+(\nabla\bfu)^{\top}]\bfn&= {\bf0}
&&\textmd{  on }\; \Gamma_N\times\nt, \label{eq7} \\
\bfu(0) &= \bfu_0 &&\textmd{  in }\; \om. \label{eq8}
\end{align}
Here $\bff$ is a body force and $\bfu_0$ describes an initial
velocity. We assume that functions $\bfu$, $\mathcal{P}$, $\bff$ and
$\bfu_0$ are smooth enough and the compatibility conditions
$\bfu_0=\bf0$ on $\Gamma_D$ and $\bfu_0 \cdot \bfn=0$ on $\Gamma_G$
hold. For simplicity we suppose that $\nu=1$ throughout the paper.

\medskip

We can formulate our problem:

\medskip

Suppose that $\bff\in L^2(0,T;\; H)$ and $\unu\in
\mathcal{D}${\footnote{The requirement $\unu\in \mathcal{D}$
represents an implicit compatibility condition imposed on the
initial data.}}. Find $\bfu\in L^2(0,T; \; \mathcal{D}) \cap
L^{\infty}(0,T;\; V)$, $\bfu_t\in L^2(0,T;\; H)$ such that
\begin{equation}\label{var_form_1a}
(\bfu_t,\bfv) + ((\bfu,\bfv)) +b(\bfu,\bfu,\bfv) = (\bff,\bfv)
\end{equation}
for every $\bfv\in V$ and for almost every $t\in(0,T)$ and
\begin{equation}\label{var_form_1b}
 \bfu(0) = \unu.
\end{equation}

\bigskip

The main difficulties of problem
\ref{var_form_1a}--\ref{var_form_1b} consist in the fact that,
because of the artificial boundary condition \ref{eq7}, we cannot
prove that $b(\bfu,\bfu,\bfu)=0$. Consequently, we are not able to
show that the kinetic energy of the fluid is controlled by the data
of the problem and the solutions of
\ref{var_form_1a}--\ref{var_form_1b} need not satisfy the energy
inequality. This is due to the fact that some uncontrolled
``backward flow'' can take place at the open parts $\Gamma_N$ of the
domain $\Omega$ and one is not able to prove global (in time)
existence results. In \cite{KraNeu1}--\cite{KraNeu5}, Kra\v cmar and
Neustupa prescribed an additional condition on the output (which
bounds the kinetic energy of the backward flow) and formulated
steady and evolutionary Navier--Stokes problems by means of
appropriate variational inequalities. In \cite{KuceraSkalak1998},
Ku\v cera and Skal\' ak prove the local--in--time existence and
uniqueness of a ``weak'' solution of the non--steady Navier--Stokes
problem with boundary condition \ref{eq7} on the part of the
boundary $\partial\Omega$, such that
\begin{equation}\label{reg_kucera_skalak}
\bfu_t\in L^2(0,T^*; \; V),\quad \bfu_{tt}\in L^2(0,T^*;\; V^*),
\quad  0<T^*\leq T,
\end{equation}
under some smoothness restrictions on $\bfu_0$ and $\mathcal{P}$. In
\cite{Kucera2009}, Ku\v{c}era supposed that the problem is solvable
in suitable function class with some given data (the initial
velocity and the right hand side). The author proved that there
exists a unique solution for data which are small perturbations of
the previous ones.

In \cite{BeKuc2007}, Bene\v{s} and Ku\v{c}era proved local existence
of solutions to the Navier--Stokes system with the so called ``do
nothing'' boundary condition
\begin{displaymath}
-{ \mathcal{P}}{\bfn} + {{\partial {\bfu}}\over{\partial{\bfn}}} =
{\bf0} \quad \textmd{ on }\Gamma_N\times(0,T)
\end{displaymath}
for sufficiently smooth data in a stronger (spatially) sense than
\ref{reg_kucera_skalak} without higher regularity with respect to
time. However, the authors excluded the boundary condition
\ref{eq6b} and proved the local existence solution solely in two
dimensions.

In the present paper, we shall prove local existence and uniqueness
solution to \ref{eq3}--\ref{eq8} such that i.a. $\bfu\in L^2(0,T^*;
\; \mathcal{D})$, $\mathcal{D}\hookrightarrow W^{2,2}(\Omega)^3$,
which is strong in the sense that the solutions possess second
spatial derivatives. The key embedding $\mathcal{D}\hookrightarrow
W^{2,2}(\Omega)^3$ is a consequence of assumptions setting on the
domain $\Omega$ and the regularity theory for the steady Stokes
system in non--smooth domains, see \cite[Corollary 4.2]{OrSan} and
\cite{MazRoss2005, MazRoss2007}.

\bigskip

In next Section \ref{Auxiliary results} we present some auxiliary
results needed in the proof of the main result stated and proved in
Section \ref{sec_main_result}.

% ----------------------------------------------------------------
\section{Auxiliary results}\label{Auxiliary results}

\begin{theorem}[Linearized problem]\label{linearized_problem}
Let $\bff\in L^2(0,T;\; H)$. Then there exists unique function
$\bfu\in L^2(0,T;\; \mathcal{D}) \cap L^{\infty}(0,T;\; V)$,
$\bfu_t\in L^2(0,T;\; H)$, such that
\begin{equation} \label{lin_var_form_1a}
(\bfu_t,\bfv) + ((\bfu,\bfv)) = (\bff,\bfv)
\end{equation}
holds for every $\bfv\in V$ and for almost every $t\in(0,T)$ and
\begin{equation} \label{lin_var_form_1b}
\bfu(0) = \bfzero.
\end{equation}
Moreover
\begin{equation}\label{eq11a}
\|\bfu_t\|_{L^2(0,T;\;{H})} +\|\bfu\|_{L^2(0,T;\;{D})}
+\|\bfu\|_{L^{\infty}(0,T;\; V)} \leq c \|\bff\|_{L^2(0,T;\;{H})},
\end{equation}
where $c = c(\om)$.
\end{theorem}
\begin{proof}
The proof is essentially the same as the proof of Theorem 2.1 in
\cite{BeKuc2007}.
\end{proof}

 The following result was established by Aubin (see
\cite{Aubin1963}).
\begin{theorem}[Aubin]\label{aubin_compact}
Let $\mathcal{B}_0$, $\mathcal{B}$, $\mathcal{B}_1$ be three Banach
spaces where $\mathcal{B}_0$, $\mathcal{B}_1$ are reflexive. Suppose
that $\mathcal{B}_0$ is continuously imbedded into $\mathcal{B}$,
which is also continuously imbedded into $\mathcal{B}_1$, and
imbedding from $\mathcal{B}_0$ into $\mathcal{B}$ is compact. For
any given $p_0$, $p_1$ with $1<p_0,p_1<\infty$, let
\begin{displaymath}
\mathcal{W}:=\left\{v \; |\; v\in L^{p_0}(0,T;\mathcal{B}_0), \;
v_t\in L^{p_1}(0,T;\mathcal{B}_1)  \right\}.
\end{displaymath}
Then the imbedding from $\mathcal{W}$ into
$L^{p_0}(0,T;\mathcal{B})$ is compact.
\end{theorem}

Let us introduce the following reflexive Banach spaces
\begin{displaymath}
X_{T} := \bigl\{\bfphi\;|\; \bfphi \in L^4(0,T;\,W^{11/8,2}(\om)^3)
\cap L^{8}(0,T;\,W^{1,24/11}(\om)^3)\bigr\}
\end{displaymath}
and
\begin{displaymath}
Y_{T} := \left\{\bfpsi\;|\; \bfpsi\in L^2(0,T;\,\mathcal{D}),\;
\bfpsi'\in L^2(0,T;\,H) \right\},
\end{displaymath}
respectively, with norms
\begin{displaymath}
\|\bfphi\|_{X_{T}} := \dc\bfphi\dc_{L^4(0,T;\,W^{11/8,2}(\om)^3)} +
\dc\bfphi\dc_{L^{8}(0,T;\,W^{1,24/11}(\om)^3)}
\end{displaymath}
and
\begin{displaymath}
\|\bfpsi\|_{Y_{T}} := \dc\bfpsi\dc_{L^2(0,T;\,\mathcal{D})} +
\dc\bfpsi'\dc_{L^2(0,T;\,H)}.
\end{displaymath}

\smallskip

Let us present some properties of $X_T$ and $Y_T$. First note that
\cite{MazRoss2005, MazRoss2007}, \cite[Corollary 4.4]{OrSan} and the
requirements imposed on the domain $\Omega$ (see Subsection
\ref{domain}) yield
\begin{equation}\label{emb_100}
\mathcal{D}\hookrightarrow W^{2,2}(\Omega)^3,
\end{equation}
which implies \cite{LiMa}
\begin{equation}\label{emb_130}
Y_{T} \hookrightarrow L^{\infty}(0,T;W^{1,2}(\om)^3).
\end{equation}
Let $\bfphi \in Y_{T}$. Raising and integrating the interpolation
inequality
\begin{equation*}
\| \bfphi(t) \|_{W^{3/2,2}(\om)^3} \leq c\| \bfphi(t)
\|^{1/2}_{W^{1,2}(\om)^3} \|\bfphi(t)\|^{1/2}_{W^{2,2}(\om)^3}
\end{equation*}
from $0$ to $T$ we get
\begin{eqnarray}\label{est_121}
 \left( \int^T_0 \| \bfphi(t) \|^4_{W^{3/2,2}(\om)^3}
 {\rm d}t \right)^{1/4}  &\leq& c\left( \int^T_0 \| \bfphi(t)
\|^2_{W^{2,2}(\om)^3}\|\bfphi(t)\|^2_{W^{1,2}(\om)^3} {\rm d}t
\right)^{1/4}
\nonumber\\
&\leq& c\| \bfphi \|^{1/2}_{{L}^{2}(0,T;W^{2,2}(\om)^3)} \| \bfphi
\|^{1/2}_{L^{\infty}(0,T;W^{1,2}(\om)^3)}
\nonumber\\
&\leq& c \, \| \bfphi \|_{Y_{T}},
\end{eqnarray}
where $c=c(\Omega)$. Hence we have
\begin{equation}\label{emb_160}
Y_{T} \hookrightarrow L^{4}(0,T;\,W^{3/2,2}(\om)^3).
\end{equation}
Using embeddings
\begin{equation*}
W^{3/2,2}(\om)^3\hookrightarrow\hookrightarrow
W^{11/8,2}(\om)^3\hookrightarrow L^2(\om)^3
\end{equation*}
and
\begin{equation*}
W^{11/8,2}(\om)^3\hookrightarrow W^{1,8/3}(\om)^3\hookrightarrow
L^{24}(\om)^3
\end{equation*}
and applying Theorem \ref{aubin_compact} we get the compact
embedding
\begin{equation}\label{comp_emb_14}
Y_{T}  \hookrightarrow \hookrightarrow L^4(0,T;\,W^{11/8,2}(\om)^3)
\hookrightarrow L^4(0,T;\,L^{24}(\om)^3).
\end{equation}
Further, raising and integrating the interpolation inequality (cf.
\cite[Theorem 5.2]{AdamsFournier1992})
\begin{equation*}
\| \bfphi(t) \|_{{W^{5/4,2}}(\om)^3} \leq c\|
\bfphi(t)\|^{1/4}_{{W^{2,2}}(\om)^3} \| \bfphi(t)
\|^{3/4}_{{W^{1,2}}(\om)^3}
\end{equation*}
from $0$ to $T$ we get
\begin{eqnarray}\label{est_101}
 \left( \int^T_0 \| \bfphi(t) \|^8_{W^{5/4,2}(\om)^3}
 {\rm d}t \right)^{1/8}  &\leq& c\left( \int^T_0 \| \bfphi(t)
\|^2_{{W^{2,2}}(\om)^3}\|\bfphi(t)\|^6_{{W^{1,2}}(\om)^3} {\rm d}t
\right)^{1/8}
\nonumber\\
&\leq& c\| \bfphi \|^{1/4}_{{L}^{2}(0,T;W^{2,2}(\om)^3)} \| \bfphi
\|^{3/4}_{L^{\infty}(0,T;W^{1,2}(\om)^3)}
\nonumber\\
&\leq& c \, \| \bfphi \|_{Y_{T}},
\end{eqnarray}
where $c=c(\Omega)$. Hence
\begin{equation}\label{emb_150}
Y_{T} \hookrightarrow L^{8}(0,T;\,W^{5/4,2}(\om)^3).
\end{equation}
Note that
\begin{equation}\label{emb_151}
W^{5/4,2}(\om)^3\hookrightarrow\hookrightarrow W^{9/8,2}(\om)^3
\hookrightarrow  W^{1,24/11}(\om)^3.
\end{equation}
Now \ref{emb_150} and \ref{emb_151} and Theorem \ref{aubin_compact}
yield the compact embedding
\begin{equation}\label{comp_emb_15}
Y_{T} \hookrightarrow \hookrightarrow L^8(0,T;\,W^{1,24/11}(\om)^3).
\end{equation}
Finally, \ref{comp_emb_14} and \ref{comp_emb_15} imply the compact
embedding
\begin{equation}\label{comp_emb_16}
Y_{T}\hookrightarrow\hookrightarrow X_{T}.
\end{equation}

\bigskip

% ----------------------------------------------------------------
\section{Main result}\label{sec_main_result}

\subsection{Statement of the result}
The main result of the paper is the following
\begin{theorem}[Main result]\label{main_result}
There exists $\ts \in (0,T]$ and the uniquely determined function
$\bfu\in L^2(0,\ts;\, \mathcal{D}) \cap L^{\infty}(0,T^*;\; V)$,
$\bfu_t\in L^2(0,\ts;\; H)$, such that $\bfu$ satisfies
\ref{var_form_1a}--\ref{var_form_1b} for every $\bfv\in{V}$ and for
almost every $t\in(0,\ts)$, where $\bff\in L^2(0,\ts;\;{H})$ and
$\unu\in\mathcal{D}$.
\end{theorem}

% ----------------------------------------------------------------

\subsection{Proof of the main result}

\subsubsection{Existence}

\begin{remark}\label{subst_hom_problem}
Setting $\bfw=\bfu  - \unu$ this amounts to solving the problem with
the homogeneous initial condition
\begin{eqnarray}
(\bfw_t,\bfv) + ((\unu+\bfw,\bfv))
+b(\unu+\bfw,\unu+\bfw,\bfv) &=& (\bff,\bfv) \label{var_form_2a}\\
\bfw(0) &=& \bfzero \label{var_form_2b}
\end{eqnarray}
for every $\bfv\in V$ and for almost every $t\in(0,T)$, where
$\bff\in L^2(0,T;\; H)$ and $\unu\in \mathcal{D}$.
\end{remark}

Denote by $B_R(T) \subset X_{T}$ the closed ball
\begin{equation}
B_R(T)  := \{\bfvarphi\in X_{T};\, \|\bfvarphi\|_{{X}_{T}}\leq R\}.
\end{equation}
For arbitrary fixed $\widetilde{\bfw}\in{X}_{T}$  we now consider
the linear problem
\begin{eqnarray}
(\bfw',\bfv) + ((\bfw,\bfv))  &=& (\bff,\bfv) - ((\unu,\bfv)) -
b(\unu,\unu,\bfv)-b(\unu,\widetilde{\bfw},\bfv)  \label{var_form_5a}\\
&&  -b(\widetilde{\bfw},\unu,\bfv)
-b(\widetilde{\bfw},\widetilde{\bfw},\bfv)
   \nonumber  \\
\bfw(0) &=& \bfzero \label{var_form_5b}
\end{eqnarray}
for every $\bfv\in {V}$ and for almost every $t\in(0,T)$.

\begin{definition}\label{def_F}
Let $\bfmathcalF:{X}_{T} \rightarrow L^2(0,T;H)$ be an operator such
that
\begin{eqnarray}\label{eq35}
(\bfmathcalF(\bfphi),\bfv)  &=& (\bff,\bfv) - ((\unu,\bfv)) -
b(\unu,\unu,\bfv) \\
&& -b(\unu,\bfphi,\bfv) -b(\bfphi,\unu,\bfv) -b(\bfphi,\bfphi,\bfv)
\nonumber
\end{eqnarray}
for every $\bfv\in {V}$ and for almost every $t\in(0,T)$.
\end{definition}

From Theorem \ref{linearized_problem} we deduce that for arbitrary
fixed $\widetilde{\bfw}\in{X}_{T}$ there exists $\bfw\in
L^2(0,T;\;\mathcal{D}) \cap L^{\infty}(0,T;\;V)$, $\bfw'\in
L^2(0,T;H)$, $\bfw(0)=\bfzero$ and
\begin{equation}
\|\bfw\|_{{X}_{T}}\leq c
\|\bfmathcalF(\widetilde{\bfw})\|_{L^2(0,T;{H})}.
\end{equation}
We now prove the following
\begin{lemma}\label{sup_lem_10}
$\bfmathcalF$ is a continuous operator from ${X}_{T}$ into
$L^2(0,T;\;{H})$ and for all $R>0$, $T>0$, and for all
$\widetilde{\bfw}\in B_R(T)$ we have
\begin{equation}\label{est_F}
\|\bfmathcalF(\widetilde{\bfw})\|_{L^2(0,T;\;{H})}\leq C_0(T) + C_1
(T^{1/8}R^2+T^{1/4}R),
\end{equation}
where $C_0(T)\rightarrow 0_+$ for $T\rightarrow 0_+$ and $C_1$ is
independent of $T$.
\end{lemma}
\begin{proof}
Obviously, there exists $C_0(T)>0$ such that
\begin{equation}\label{eq37a}
\|(\bff,\cdot) - ((\unu,\cdot)) -
b(\unu,\unu,\cdot)\|_{L^2(0,T;{H})}\leq C_0(T),
\end{equation}
$C_0(T)\rightarrow 0_+$ for $T\rightarrow 0_+$.

Using the interpolation inequality one obtains
\begin{eqnarray}\label{eq38}
\| b(\widetilde{\bfw},\unu,\cdot)\|_{L^2(0,T;{H})}&\leq&
\left(\int_0^T\|\unu\|^2_{W^{1,4}(\om)^3}\|\widetilde{\bfw}\|_{L^4(\om)^3}^2
{\rm d}t\right)^{1/2}\\
\nonumber &\leq&
\|\unu\|_{W^{1,4}(\om)^3}\|\widetilde{\bfw}\|_{L^2(0,T;L^{4}(\om)^3)}\\
\nonumber &\leq& T^{1/4} \|\unu\|_{W^{1,4}(\om)^3}
\|\widetilde{\bfw}\|_{L^{4}(0,T;L^{4}(\om)^3)}\\
\nonumber &\leq& c \,
T^{1/4}\|\unu\|_{\mathcal{D}}\|\widetilde{\bfw}\|_{{X}_{T}}.
\end{eqnarray}
Similarly, we obtain the inequalities
\begin{eqnarray}\label{eq39}
\| b(\unu,\widetilde{\bfw},\cdot)\|_{L^2(0,T;{H})}&\leq&
\left(\int_0^T\|\unu\|^2_{L^{\infty}(\Omega)^3}\|\widetilde{\bfw}\|_{W^{1,2}(\om)^3}^2
{\rm d}t\right)^{1/2}\\
\nonumber &\leq&
\|\unu\|_{L^{\infty}(\Omega)^3}\|\widetilde{\bfw}\|_{L^2(0,T;W^{1,2}(\om)^3)}\\
\nonumber &\leq& T^{1/4} \|\unu\|_{L^{\infty}(\Omega)^3}
\|\widetilde{\bfw}\|_{L^{4}(0,T;W^{1,2}(\om)^3)}\\
\nonumber &\leq& c \,
T^{1/4}\|\unu\|_{\mathcal{D}}\|\widetilde{\bfw}\|_{{X}_{T}}
\end{eqnarray}
and
\begin{eqnarray}\label{eq40}
\quad \|b(\widetilde{\bfw},\widetilde{\bfw},\cdot)\|_{L^2(0,T;{H})}
&\leq& \left(\int_0^T \|\widetilde{\bfw}\|^2_{L^{24}(\om)^3}
\|\widetilde{\bfw}\|^2_{W^{1,24/11}(\om)^3}{\rm d}t\right)^{1/2} \\
\nonumber &\leq& T^{1/8}
\|\widetilde{\bfw}\|_{L^4(0,T;L^{24}(\om)^3)}
\|\widetilde{\bfw}\|_{L^8(0,T;W^{1,24/11}(\om)^3)}\\
\nonumber &\leq& c \, T^{1/8}\|\widetilde{\bfw}\|^2_{{X}_{T}}.
\end{eqnarray}
The inequalities \ref{eq37a}--\ref{eq40} yield
$\bfmathcalF(\widetilde{\bfw})\in\ldh$ and the inequality
\ref{est_F} holds.

Let $\widetilde{\bfw}_1,\widetilde{\bfw}_2\in X_{T}$ and ${\bfz}=
\widetilde{\bfw}_2 - \widetilde{\bfw}_1$. Then
\begin{multline}
\|\bfmathcalF(\widetilde{\bfw}_2)-\bfmathcalF(\widetilde{\bfw}_1)\|_{L^2(0,T;{H})}
\leq \|\bfb(\bfz,\unu,\cdot)\|_{L^2(0,T;{H})}
\\
+ \|\bfb(\unu,\bfz,\cdot)\|_{L^2(0,T;{H})} +
\|\bfb(\widetilde{\bfw}_2,\bfz,\cdot)\|_{L^2(0,T;{H})} +
\|\bfb(\bfz,\widetilde{\bfw}_1,\cdot)\|_{L^2(0,T;{H})}
\end{multline}
and
\begin{eqnarray}\label{ineq41}
\|b(\widetilde{\bfw}_2,\bfz,\cdot)\|_{L^2(0,T;{H})} &\leq&
\left(\int_0^T \|\widetilde{\bfw}_2\|^2_{L^{24}(\om)^3}
\|\bfz\|^2_{W^{1,24/11}(\om)^3}{\rm d}t\right)^{1/2}
\\ \nonumber
&\leq& \|\widetilde{\bfw}_2\|_{L^4(0,T;L^{24}(\om)^3)}
\|\bfz\|_{L^4(0,T;W^{1,24/11}(\om)^3)}
\\ \nonumber
&\leq& c \, \|\widetilde{\bfw}_2\|_{{X}_{T}}\|\bfz\|_{{X}_{T}}.
\end{eqnarray}
Similarly
\begin{eqnarray}\label{ineq42}
\| b(\bfz,\widetilde{\bfw}_1,\cdot)\|_{L^2(0,T;{H})}
&\leq&\left(\int_0^T \|\bfz\|^2_{L^{24}(\om)^3}
\|\widetilde{\bfw}_1\|^2_{W^{1,24/11}(\om)^3}{\rm d}t\right)^{1/2}
\\ \nonumber
&\leq&  c \, \|\bfz\|_{{X}_{T}}\|\widetilde{\bfw}_1\|_{{X}_{T}}.
\end{eqnarray}

Inequalities \ref{ineq41}--\ref{ineq42} and \ref{eq38}--\ref{eq39}
imply that $\bfmathcalF$ is a continuous operator from ${X}_{T}$
into $L^2(0,T;{H})$.
\end{proof}

The proof of the main result is based on the Brouwer fixed point
theorem. Let the operator $\bfmathcalA$ be defined as follows. Given
a function $\widetilde{\bfw} \in X_T$, consider the linear problem
\begin{eqnarray}
(\bfw',\bfv) + ((\bfw,\bfv))  &=&
(\bfmathcalF(\widetilde{\bfw}),\bfv)
\label{var_form_200a}\\
\bfw(0) &=& \bfzero \label{var_form_200b}
\end{eqnarray}
for every $\bfv \in V$ and for almost every $t\in(0,T)$, where
$\bfmathcalF$ is defined by Definition \ref{def_F}. Theorem
\ref{linearized_problem} and Lemma \ref{sup_lem_10} ensure that the
linear problem \ref{var_form_200a}--\ref{var_form_200b} has a unique
solution $\bfw \in Y_T$. Define $\bfmathcalA:X_T\rightarrow Y_T$ by
setting $\bfmathcalA(\widetilde{\bfw})=\bfw$. Clearly, the
inequality \ref{eq11a} and Lemma \ref{sup_lem_10} imply that
$\bfmathcalA$ is a continuous operator from $X_T$ into $Y_T$. For
all $\widetilde{\bfw}\in B_R(T)$, taking \ref{eq11a} and \ref{est_F}
together, we deduce
\begin{equation*}
\|\bfmathcalA(\widetilde{\bfw})\|_{X_T}\leq c_1
\|\bfmathcalA(\widetilde{\bfw}) \|_{Y_T} \leq c_2
\|\bfmathcalF(\widetilde{\bfw})\|_{\ldsh}  \leq c_3(T) + c_4
(T^{1/8}R^2+T^{1/4}R),
\end{equation*}
where $c_3(T)\rightarrow 0$ for $T\rightarrow 0$ and $c_1,c_2$ and
$c_4$ do not depend on $T$. Hence for $T=T^*$, $T^*>0$ sufficiently
small, and for a sufficiently large $R$, $\bfmathcalA$ maps
$B_R(T^*)$ into itself. Since $\bfmathcalA$ is a continuous operator
from $X_T$ into $Y_T$ and $Y_T\hookrightarrow \hookrightarrow X_T$,
$\bfmathcalA$ is totally continuous operator from $X_{T^*}$ into
$X_{T^*}$, where $X_{T^*}$ is a reflexive Banach space. Therefore
there exists a fixed point $\bfw\in B_R(T^*)$ such that
$\bfmathcalA(\bfw) = \bfw$ in $X_{T^*}$.

\subsubsection{Uniqueness}
Suppose that there are two solutions $\bfu_1,\bfu_2\in Y_{T^*}$ of
\ref{var_form_1a}--\ref{var_form_1b} on $(0,T^*)$. Denote
$\bfz=\bfu_1-\bfu_2$ then
\begin{equation}
(\bfz_t,\bfv) + ((\bfz,\bfv))
+b(\bfz,\bfu_2,\bfv)+b(\bfu_1,\bfz,\bfv)=0
\end{equation}
holds for all $\bfv \in V$ and almost every $t \in (0,T)$ and
$\bfz(0)={\bf0}$. Hence
\begin{eqnarray*}
\frac{1}{2}\frac{{\rm d}}{{\rm d} t}\|\bfz(t)\|^2_{H} +
\|\bfz(t)\|^2_{V} &\leq&
|b(\bfu_1(t),\bfz(t),\bfz(t))|+|b(\bfz(t),\bfu_2(t),\bfz(t))|  \\
&\leq&
\|\bfu_1(t)\|_{L^4(\Omega)^3}\|\nabla\bfz(t)\|_{L^2(\Omega)^3}
\|\bfz(t)\|_{L^4(\Omega)^3}\\
&&+\|\bfz(t)\|^2_{L^4(\Omega)^3}
\|\nabla\bfu_2(t)\|_{L^2(\Omega)^3}.
\end{eqnarray*}
Using the interpolation inequality
\begin{equation*}
\|\bfz(t)\|_{L^4(\Omega)^3} \leq c\, \|\bfz(t)\|^{3/4}_{V}
\|\bfz(t)\|^{1/4}_{L^2(\Omega)^3}
\end{equation*}
we get
\begin{multline}
\frac{1}{2}\frac{{\rm d}}{{\rm d} t}\|\bfz(t)\|^2_{H} +
\|\bfz(t)\|^2_{V}  \leq
c_1\|\bfu_1(t)\|_{L^4(\Omega)^3}\|\bfz(t)\|^{7/4}_{V}
\|\bfz(t)\|^{1/4}_{L^2(\Omega)^3} \\
+c_2\|\bfz(t)\|^{3/2}_{V}
\|\bfz(t)\|^{1/2}_{L^2(\Omega)^3}\|\bfu_2(t)\|_{W^{1,2}(\Omega)^3}.
\end{multline}
Using Young's inequality we deduce
\begin{multline}
\frac{1}{2}\frac{{\rm d}}{{\rm d} t}\|\bfz(t)\|^2_{H} +
\|\bfz(t)\|^2_{V}  \leq \delta\|\bfz(t)\|^{2}_{V}
\\
+c_{\delta}\|\bfz(t)\|^{2}_{L^2(\Omega)^3} \left(
\|\bfu_1(t)\|^8_{L^4(\Omega)^3} +
\|\bfu_2(t)\|^4_{W^{1,2}(\Omega)^3} \right),
\end{multline}
where $\delta>0$ can be chosen arbitrarily small and therefore
\begin{equation}
\frac{{\rm d}}{{\rm d} t}\|\bfz(t)\|^2_{H} \leq 2c_{\delta} \; \|
\bfz(t) \|^2_{H} \left( \|\bfu_1(t)\|^8_{L^4(\Omega)^3} +
\|\bfu_2(t)\|^4_{W^{1,2}(\Omega)^3} \right).
\end{equation}
Hence, we have the differential inequality
\begin{displaymath}
y'(t) \leq \theta(t)y(t),
\end{displaymath}
where
\begin{displaymath}
y(t) = \|\bfz(t)\|^2_{H} \quad  \textmd{ and } \quad \theta(t)=
2c_{\delta}  \left( \|\bfu_1(t)\|^8_{L^4(\Omega)^3} +
\|\bfu_2(t)\|^4_{W^{1,2}(\Omega)^3} \right) \in L^1((0,T)),
\end{displaymath}
from which we obtain, using the technique of Gronwall's lemma,
\begin{displaymath}
\frac{{\rm d}}{{\rm d} t}\left( y(t) \exp{\left(-\int_0^t\theta(s)
\;{\rm d}s \right) } \right)  \leq 0
\end{displaymath}
and
\begin{displaymath}
y(t)\leq y(0) \exp{\left(\int_0^t\theta(s) \;{\rm d}s \right) } .
\end{displaymath}
Therefore
\begin{displaymath}
\|\bfz(t)\|^2_{H} \leq \|\bfz(0)\|^2_{H} \exp{\left(\int_0^t
2c_{\delta} \left( \|\bfu_1(s)\|^8_{L^4(\Omega)^3} +
\|\bfu_2(s)\|^4_{W^{1,2}(\Omega)^3} \right) \;{\rm d}s \right) }
\end{displaymath}
for all $t \in (0,T)$. Now the uniqueness follows from the fact that
$\bfz(0)={\bf0}$.

% ----------------------------------------------------------------

\bigskip

The author thanks Prof. Petr Ku\v{c}era for numerous suggestions
that improved the presentation of this paper.

% ----------------------------------------------------------------

%For acknowledgements section, please don't number the section, please begin it with \section*{Acknowledgements}
%\section*{Acknowledgments} We would like to thank you for \textbf{following
%the instructions above} very closely in advance. It will definitely
%save us lot of time and expedite the process of your paper's
%publication.

% You may incorporate your references as follows in your main tex file.
% Using BibTex is not recommended but can be handled.


\begin{thebibliography}{99}

\bibitem{AdamsFournier1992}
\newblock A.~Adams and J.F.~Fournier.
\newblock Sobolev spaces,
\newblock Pure and Applied Mathematics 140, Academic Press, 2003.

\bibitem{Aubin1963}
\newblock J.-P.~Aubin.
\newblock \emph{Un th\'{e}or\`{e}me de compacit\'{e}},
\newblock  C.R. Acad. Sci., \textbf{256} (1963), 5042–-5044.

\bibitem{BeKuc2007}
\newblock M.~Bene\v s and P. Ku\v cera.
\newblock \emph{Non-steady Navier--Stokes equations with homogeneous mixed
boundary conditions and arbitrarily large initial condition},
\newblock  Carpathian Journal of Mathematics, \textbf{23}(2007), No. 1--2, 32--40.

\bibitem{DeKra}
\newblock P.~Deuring and S.~Kra\v cmar,
\newblock  \emph{Exterior stationary Navier--Stokes
flows in 3D with non-zero velocity at infinity approximation by
flows in bounded domains},
\newblock  Math. Nachr., \textbf{269/270}(2004), 86--115.

\bibitem{KraNeu1}
\newblock S.~Kra\v cmar  and J.~Neustupa,
\newblock  \emph{Global existence of weak solutions of a nonsteady variational inequalities of the
Navier--Stokes type with mixed boundary conditions},
\newblock Proc. of the
conference ISNA'92, August-September 1992, Part III, Publ. Centre of
the Charles Univ., Prague, 156--157.

\bibitem{KraNeu2}
\newblock S.~Kra\v cmar and J.~Neustupa,
\newblock  \emph{Modelling of flows of a viscous incompressible fluid through a
channel by means of variational inequalities},
\newblock ZAMM, \textbf{74}(1994), No. 6, 637--639.

\bibitem{KraNeu5}
\newblock S.~Kra\v cmar and J.~Neustupa,
\newblock \emph{A weak solvability of a steady
variational inequality of the Navier-Stokes type with mixed boundary
conditions},
\newblock Nonlinear Anal., \textbf{47}(2001), No. 6, 4169--4180.

\bibitem{Kucj}
\newblock P.~Ku\v cera,
\newblock \emph{Solution of the Stationary Navier-Stokes
Equations with Mixed Boundary Conditions in a Bounded Domain},
\newblock  Pitman Research Notes in Mathematics Series, \textbf{379}(1998), Longman,
127--131.

\bibitem{Kucd}
\newblock P.~Ku\v cera ,
\newblock  \emph{A structure of the set of critical points to
the Navier-Stokes equations with mixed boundary conditions},
\newblock Pitman Research Notes in Mathematics Series,
\textbf{388}(1998), Longman, 201--205.

\bibitem{Kucera2009}
\newblock P.~Ku\v cera,
\newblock \emph{Basic properties of solution of the non-steady Navier-Stokes
equations with mixed boundary conditions in a bounded domain},
\newblock Annali dell’ Universita di Ferrara, \textbf{55}(2009), 289-308.

\bibitem{KuceraSkalak1998}
\newblock P.~Ku\v cera and Z.~Skal\'ak,
\newblock {\em Solutions to the Navier-Stokes
Equations with Mixed Boundary Conditions},
\newblock Acta Applicandae Mathematicae, \textbf{54}(1998), No. 3, 275--288.

\bibitem{KufFucJoh1977}
\newblock A.~Kufner, O.~John and S.~Fu\v{c}\'{i}k,
\newblock Function Spaces,
\newblock Academia, 1977.

\bibitem{LiMa}
\newblock J.~L.~Lions, E.~Magenes,
\newblock Non-Homogeneous Boundary Value Problems and Applications,
\newblock Vols. I and II., Springer-Verlag, Berlin 1972.

\bibitem{MazRoss2003}
\newblock V.G.~Maz’ya and J.~Rossmann,
\newblock \emph{Weighted $L_p$ estimates of solutions to
boundary value problems for second order elliptic systems in
polyhedral domains},
\newblock ZAMM, \textbf{83}(2003),  No. 7, 435--467.

\bibitem{MazRoss2005}
\newblock V.G.~Maz’ya and J.~Rossmann,
\newblock \emph{Pointwise estimates for Green's kernel of
a mixed boundary value problem to the Stokes system in a polyhedral
cone},
\newblock Math. Nachr., \textbf{278}(2005), No. 15, 1766–-1810.

\bibitem{MazRoss2007}
\newblock V.G.~Maz’ya and J.~Rossmann,
\newblock \emph{$L^p$ estimates of solutions to mixed
boundary value problems for the Stokes system in polyhedral
domains},
\newblock Math. Nachr., \textbf{280}(2007), No. 7, 751–-793.

\bibitem{MazRoss2009}
\newblock  V.G.~Maz’ya and J.~Rossmann,
\newblock \emph{Mixed boundary value problems for the stationary
Navier-Stokes System in  polyhedral domains},
\newblock Arch. Rational Mech. Anal., \textbf{194}(2009), No. 2, 669--712.

\bibitem{necas1967}
\newblock J.~Ne\v{c}as,
\newblock  Les methodes directes en theorie des equations elliptiques.
Academia, Prague 1967.

\bibitem{OrSan}
\newblock M.~Orlt and A.-M.~S$\ddot{\textmd{a}}$ndig,
\newblock \emph{Regularity of viscous Navier-Stokes flows in nonsmooth domains},
\newblock Lecture Notes in Pure and Appl. Math., \textbf{167}(1993), 185--201.

\bibitem{sohr2001}
\newblock H.~Sohr,
\newblock  The Navier-Stokes Equations, An Elementary Functional Analytic Approach,
\newblock Birkhäuser, Advanced Texts, 2001.

\bibitem{Te}
\newblock R.~Temam,
\newblock Navier-Stokes Equations, theory and numerical analysis,
\newblock  American Mathematical Society, 2001.


\end{thebibliography}
\end{document}